\documentclass{amsart}
\usepackage{cases}
\usepackage{amsmath}
\usepackage{amssymb}
\usepackage{amsfonts}
\usepackage{hyperref}
\newtheorem{theorem}{Theorem}[section]

\theoremstyle{definition}

\newtheorem{remark}[theorem]{Remark}
\numberwithin{equation}{section}
\begin{document}
\title[Upper bounds on the first eigenvalue via
Bakry-\'{E}mery curvature II]% end with percent
{Upper bounds on the first eigenvalue for a diffusion operator via
Bakry-\'{E}mery Ricci curvature II}
\author{Jia-Yong Wu}

\address{Department of Mathematics, Shanghai Maritime University,
Haigang Avenue 1550, Shanghai 201306, P. R. China}

\email{jywu81@yahoo.com}

\thanks{This work is partially supported by the NSFC (No. 11101267)
and the Science and Technology Program of Shanghai Maritime
University (No. 20120061).}


\dedicatory{}

\begin{abstract}
Let $L=\Delta-\nabla\varphi\cdot\nabla$ be a symmetric diffusion
operator with an invariant measure $d\mu=e^{-\varphi}dx$ on a
complete Riemannian manifold. In this paper we prove Li-Yau gradient
estimates for weighted elliptic equations on the complete manifold
with $|\nabla \varphi|\leq\theta$ and $\infty$-dimensional
Bakry-\'{E}mery Ricci curvature bounded below by some negative
constant. Based on this, we give an upper bound on the first
eigenvalue of the diffusion operator $L$ on this kind manifold,
and thereby generalize a Cheng's result on the Laplacian case
(Math. Z., 143 (1975) 289-297).
\end{abstract}
\maketitle

\section{Introduction and main result}
Given $(M,g)$ be an $n$-dimensional complete Riemannian manifold
with the Ricci curvature satisfying $Ric(g)\geq-(n-1)$,
Cheng (Theorem 4.2 in \cite{[Cheng]})
proved an upper bound of the first nontrivial eigenvalue of the
Laplacian
\[
\lambda_1\leq \frac{(n-1)^2}{4}.
\]

Later, the author \cite{[Wu]} extended this result to the case of
the diffusion operator with the $m$-dimensional Bakry-\'{E}mery
Ricci curvature bounded below.

\vspace{0.5em}

\noindent \textbf{Theorem A.} (Wu \cite{[Wu]}) \emph{Let
$(M^n,g)$ be an $n$-dimensional ($n\geq2$) complete Riemannian
manifold and $\varphi \in C^2(M^n)$ be a function. Assume that the
$m$-dimensional Bakry-\'{E}mery Ricci curvature satisfies
\[
Ric_{m,n}(L)\geq-(n-1).
\]
Then we have an upper bound estimate on the first nontrivial
eigenvalue for the diffusion operator
$L=\Delta-\nabla\varphi\cdot\nabla$
\[
\lambda_1\leq\frac{(m-1)(n-1)}{4}.
\]}

\vspace{0.5em}

\noindent In Theorem A, the $m$-dimensional Bakry-\'{E}mery Ricci
curvature (see \cite{[BE],[BQ1],[BQ2],[LD]}) is defined by
\[
Ric_{m,n}(L):=Ric +Hess(\varphi)-\frac{\nabla \varphi \otimes \nabla
\varphi}{m-n},
\]
where $Ric$ and $Hess$ denote the Ricci curvature and the Hessian of
the metric $g$, respectively, and where $m:=\mathrm{dim}_{BE}(L)\geq
n$ is called the Bakry-\'{E}mery dimension of $L$, which is a constant
and $m=n$ if and only if $\varphi$ is constant. When $\varphi$ is
constant, $Ric_{m,n}(L)=Ric$ and Theorem A reduces to the Cheng's
theorem by taking $m=n$. The essential tool in proving Theorem A is the
Li-Yau's gradient estimate trick \cite{[Li-Yau1]}, which are originated
by Yau \cite{[Yau]} (see also Cheng and Yau
\cite{[Cheng-Yau]}).

The $m$-dimensional Bakry-\'{E}mery Ricci curvature, mentioned
in Theorem A, has a natural extension to non-smooth metric measure
spaces \cite{[Lo-Vi],[St1],[St2]}. We refer to the survey
article \cite{[WW2]} for details on this tensor.
A remarkable feature for $m$-dimensional Bakry-\'{E}mery Ricci
curvature is that Laplacian comparison theorems hold for
$Ric_{m,n}(L)$ that look like the case of Ricci tensor in a
$m$-dimensional manifold \cite{[WW]} (see also
\cite{[FLZ],[LD],[Wu]}). This is why many geometric and topological
results for manifolds with a lower bound on the Ricci curvature can
be easily extended to smooth metric measure spaces with only a
$m$-dimensional Bakry-\'{E}mery Ricci curvature bounded below. For
discussions of this use, see for example Theorem 1.3 in
\cite{[FLZ]}, Theorem 5.1 in \cite{[LD]} and Theorem 1.1 in
\cite{[Wu]}.

For the $m$-dimensional Bakry-\'{E}mery Ricci curvature, we can
allow $m$ to be infinite. In this case this becomes
\[
Ric(L):=\lim_{m\to +\infty}Ric_{m,n}(L)=Ric+Hess(\varphi),
\]
which is called the $\infty$-dimensional Bakry-\'{E}mery Ricci
curvature. Bakry and \'{E}mery \cite{[BE]} extensively studied this
tensor and its relationship to diffusion processes. The
$\infty$-dimensional Bakry-\'{E}mery Ricci tensor also occurs
naturally in many different subjects, see e.g.
\cite{[FLZ],[FG],[Lott1],[Perelman],[Wylie]}. Recently, the
$\infty$-dimensional Bakry-\'{E}mery Ricci curvature has become an
important object of study in Riemannian geometry, in large part due
to the gradient Ricci soliton equation:
\begin{equation}\label{soliton}
Ric(L)=\rho g
\end{equation}
for some constant $\rho$, which plays an important role in the
Ricci flow. For the recent progress on Ricci solitons the
reader may refer to Cao's survey papers \cite{[Cao1],[Cao2]}.

As mentioned in Remark 1.2 (2) of \cite{[Wu]}, in the above
Theorem A, if we let $m=\infty$, then $\frac{(m-1)(n-1)}{4}=\infty$.
At this time, we do not obtain any useful upper bound on the
first eigenvalue when $Ric(L)\geq-(n-1)$ from Theorem A.
Naturally, we ask if we can extend Cheng's theorem to the case of
the diffusion operator $L$ with only the $\infty$-dimensional
Bakry-\'{E}mery Ricci curvature bounded below. This is not an
easy question to answer. One main difficulty is that we can not get a
proper local gradient estimate on Riemannian manifolds with only a
lower $\infty$-dimensional Bakry-\'{E}mery Ricci curvature bound,
since local gradient estimates are closely related to Laplacian
comparison theorems. However, classical Laplacian comparison
theorems can not be directly extended to the case for only a
lower bound on the $\infty$-dimensional Bakry-\'{E}mery Ricci
tensor. This is a main difference compared with the
$m$-dimensional Bakry-\'{E}mery tensor.

As many recent authors said, when dealing with the
$\infty$-dimensional Bakry-Emery Ricci curvature, many geometric and
topological results remain true under some assumptions of the
potential functions $\varphi$ (see e.g. \cite{[FLZ],[WW2],[WW]}).
Following this idea, in this paper if we assume additionally that
the functions $|\nabla \varphi|$ are bounded, we can derive the
desired gradient estimates under the $\infty$-dimensional
Bakry-\'{E}mery Ricci curvature bounded below, analogous to
\cite{[FLZ],[WW],[Wu]}. Furthermore, we can derive the following
useful upper bound on the first eigenvalue of $L$.
\begin{theorem}\label{mainthm}
Let $(M^n,g)$ be an $n$-dimensional ($n\geq2$) complete Riemannian
manifold and $\varphi \in C^2(M^n)$ be a function satisfying
$|\nabla\varphi|\leq\theta$ for some constant $\theta\geq 0$.
Assume that
\begin{equation}\label{ass1}
Ric(L)\geq-(n-1).
\end{equation}
Then we have an upper bound estimate
on the first nontrivial eigenvalue for the diffusion operator
$L=\Delta-\nabla\varphi\cdot\nabla$
\begin{equation}\label{equth14}
\lambda_1\leq\frac{1}{4}(n-1+\theta)^2.
\end{equation}
\end{theorem}
When $\varphi$ is constant, we can take $\theta=0$, and our theorem reduces to
Cheng's theorem \cite{[Cheng]}. The assumption $|\nabla\varphi|\leq\theta$ guarantees
that the weighted comparison theorem \cite{[WW]} holds when $Ric(L)$ is
bounded below. Su-Zhang \cite{[SuZh]} provided a simple example to explain
that the assumption $|\nabla\varphi|$ is necessary. Recently, Munteanu-Wang
\cite{[MuWa2]} and Su-Zhang \cite{[SuZh]} also independently obtained
this eigenvalue estimate by means of the weighted volume comparison
theorem and the eigenvalue variational principle. We emphasize that
our proof relies on Li-Yau's gradient estimates (see Theorem \ref{L301}
in Sect. \ref{sec2}).

\begin{remark}\label{rem11}
We remark that there exists an example which shows that estimate
\eqref{equth14} is sharp. Consider $M=\mathbb{R}\times N^{n-1}$ endowed
with the warped product metric
\[
ds_M^2=dt^2+\exp(-2t)ds_N^2.
\]
If $\{\bar{e}_\alpha\}$ for $\alpha=2,...,n$ form an orthonormal basis of
the tangent space of $N$, then $e_1=\frac{\partial}{\partial t}$ together
with $\{e_\alpha=\exp(-t)\bar{e}_\alpha\}$ form an orthonormal basis for
the tangent space of $M$. If $\varphi(t,x)=\theta t$ for
$(t,x)\in\mathbb{R}\times N^{n-1}$, by the standard computation, then the
$\infty$-dimensional Bakry-\'{E}mery Ricci curvature of $M$ is
\[
Ric(L)_{1j}=Ric_{M,1j}+\varphi_{1j}
=-(n-1)\delta_{1j}
\]
and
\[
Ric(L)_{\alpha\beta}=\exp(2t)Ric_{N,\alpha\beta}-(n-1)\delta_{\alpha\beta}.
\]
Clearly, $Ric_N\geq 0$ if and only if $Ric(L)\geq-(n-1)$.
Moreover, we \emph{claim} that $\lambda_1=\frac{(n-1+\theta)^2}{4}$.
Indeed, we choose function $f=\exp\left(\frac{n-1+\theta}{2}t\right)$,
and have that
\[
\Delta_\varphi f=\frac{d^2f}{dt^2}-(n-1)\frac{df}{dt}
-\frac{d\varphi}{dt}\cdot\frac{df}{dt}
=-\frac{(n-1+\theta)^2}{4}f,
\]
since $\Delta=\frac{\partial^2}{\partial t^2}-(n-1)\frac{\partial}{\partial t}
+\exp(2t)\Delta_N$. From this we conclude that
$\lambda_1=\frac{(n-1+\theta)^2}{4}$ as claimed.
\end{remark}

\begin{remark}\label{rem11a}
If $Ric(L)\equiv-(n-1)$ with $|\nabla\varphi|\leq\theta$,
then $M$ must be Einstein by \cite{[Pigoli]}. If $M$ is the gradient expanding Ricci soliton:
$Ric+\nabla\nabla\varphi=\rho g$, $\rho<0$
(see \cite{[Ha]} or \cite{[Cao0]}), we can derive that
$|\nabla\varphi|^2=-R+2\rho\varphi$. If we assume
$\varphi\geq c$ for some constant $c\leq n/2$, then
\[
|\nabla\varphi|^2\leq-n\rho+2\rho\varphi\leq-n\rho+2c\rho,
\]
where we used $R\geq n\rho$ \cite{[Chen]}. Hence by \cite{[Pigoli]},
we conclude that $M$ is Einstein.
\end{remark}

\begin{remark}\label{rem11b}
In an earlier version of this result, the author proved that
for any constant $m_0\geq n$, the first nontrivial eigenvalue
satisfies
\[
\lambda_1\leq\frac{(m_0-1)(n-1)+\frac{m_0-1}{m_0-n}\theta^2}{4}.
\]
This estimate can be optimized by estimate \eqref{equth14},
pointed out by the referee. Since $m_0$ is arbitrary, letting
$m_0=n+k\theta$, where $k>0$ is a free parameter, then
\[
(m_0-1)(n-1)+\frac{m_0-1}{m_0-n}\theta^2
=(n-1+k\theta)(n-1+k^{-1}\theta)\geq(n-1+\theta)^2,
\]
where $m_0=n$ if and only if $\theta=0$.
\end{remark}

The structure of this paper is organized as follows. In Sect.
\ref{sec2}, we derive Li-Yau gradient estimates, i.e., Theorem
\ref{L301}. The proof makes use of the author's previous result in
\cite{[Wu]}, combining the concept of Bakry-\'{E}mery tensor
used in \cite{[LD]}. In Sect. \ref{sec5}, we apply Theorem
\ref{L301} to the setting of gradient Ricci solitons.
Finally in appendix, we give a detailed proof of Theorem
\ref{L301} though the proof method nearly follows from that
of Theorem 2.1 in \cite{[Wu]}. We include it because we
feel it might be useful in other applications.

% ------------------------------------------------------------------------

\section{Some basic gradient estimates}\label{sec2}
In this section, we will prove an important gradient estimate (see
\eqref{lemm1b} below), which also implies the proof of Theorem
\ref{mainthm}.

\begin{theorem}\label{L301}
Let $(M^n,g)$ be an $n$-dimensional ($n\geq2$) complete Riemannian
manifold and $\varphi \in C^2(M^n)$ be a function satisfying
$|\nabla \varphi|\leq\theta$ for some constant $\theta\geq 0$.
Assume that
\[
Ric(L)\geq-(n-1)K
\]
for some constant $K\geq 0$. Let $f$ be a positive function defined
on the geodesic ball $B_p(2R)\subset M^n$ satisfying
\[
Lf=-\lambda f
\]
for some constant $\lambda\geq 0$. Then for any constant $m_0\geq
n$, there exists a constant $C$ depending on $m_0$ and $n$ such
that
\begin{equation}\label{lemm1a}
\frac{|\nabla f(x)|^2}{f^2(x)} \leq
\frac{[2(m_0-1)(n-1)+\epsilon]K}{(2-\epsilon)}
+\frac{2(m_0-1)\theta^2}{(m_0-n)(2-\epsilon)}
+C\left(\frac{1+\epsilon^{-1}}{R^2}+\lambda\right)
\end{equation}
for all $x\in B_p(R)$ and for any $0<\epsilon<2$.
Furthermore, if the positive function $f$ is defined on $(M^n,g)$,
then for any constant $m_0\geq n$, we have
\begin{equation}
\begin{aligned}\label{lemm1b}
|\nabla \ln f|^2&\leq\frac{(m_0-1)(n-1)K}{2}-\lambda
+\frac{m_0-1}{2(m_0-n)}\theta^2\\
&\quad+\sqrt{\frac{\left[(m_0-1)(n-1)K
+\frac{m_0-1}{m_0-n}\theta^2\right]^2}{4}
-\left[(m_0-1)(n-1)K+\frac{m_0-1}{m_0-n}\theta^2\right]\lambda}
\end{aligned}
\end{equation}
and
\begin{equation}\label{eigeestup}
\lambda\leq\frac{(m_0-1)(n-1)K+\frac{m_0-1}{m_0-n}\theta^2}{4}.
\end{equation}
\end{theorem}

\begin{remark}\label{re22}(1) Similar to Remark \ref{rem11b}, letting
$m_0=n+k\theta$, where $k>0$ is a free parameter, then
\begin{equation*}
\begin{aligned}
(m_0-1)(n-1)K+\frac{m_0-1}{m_0-n}\theta^2
&=\left[(n-1)\sqrt{K}+k\sqrt{K}\theta\right]
\left[(n-1)\sqrt{K}+\frac{\theta}{k\sqrt{K}}\right]\\
&\geq\left[(n-1)\sqrt{K}+\theta\right]^2,
\end{aligned}
\end{equation*}
where $m_0=n$ if and only if $\theta=0$. Therefore estimates
\eqref{lemm1b} and \eqref{eigeestup} can be respectively
optimized by
\[
|\nabla \ln f|^2\leq\frac{\left[(n-1)\sqrt{K}+\theta\right]^2}{2}-\lambda
+\sqrt{\frac{\left[(n-1)\sqrt{K}+\theta\right]^4}{4}-\left[(n-1)\sqrt{K}+\theta\right]^2\lambda}
\]
and
\[
\lambda\leq\frac{\left[(n-1)\sqrt{K}+\theta\right]^2}{4}.
\]
\end{remark}

\begin{remark}\label{re22b}
If $Ric(L)\geq0$, by letting $K=0$ and then letting $m_0\to\infty$ in
\eqref{eigeestup}, we immediately obtain $\lambda_1\leq\frac{\theta^2}{4}$,
which has been proved by Munteanu and Wang in \cite{[MuWa]} under
some weak assumption.
\end{remark}

\begin{remark}\label{re22b}
Our Li-Yau gradient estimates can be used to prove splitting type theorems for
complete manifolds with $\infty$-dimensional Bakry-\'{E}mery Ricci
curvature. This was done by the author in a separated paper \cite{[Wusp]}.
Recently, Munteanu and Wang \cite{[MuWa2]} obtained similar gradient
estimates for weighted harmonic functions ($\lambda=0$ in Theorem \ref{L301})
under some oscillation of function $\varphi$. Their proof is a mixture of
the Bochner identity and the DeGiorgi-Nash-Moser theory.
\end{remark}

\begin{proof}[Proof of Theorem \ref{L301}]
We moved our original proof using the Li-Yau gradient estimate method
to the appendix because we feel it seems to be tedious.
Here we use a simple and direct proof, which was pointed by the referee.
We are very grateful to the referee for his valuable comment which
leads us to give this quick proof. According to Section 1.5 in
X.-D. Li's paper \cite{[LD]}, the conditions $Ric(L)\geq-(n-1)K$
and $|\nabla\varphi|\leq \theta$ imply that the $m$-dimensional
Bakry-Emery Ricci curvature is bounded from below by a new constant
for any $m_0>n$, i.e.,
\[
Ric_{m,n}(L)\geq-(n-1)\left[K+\frac{\theta^2}{(m_0-n)(n-1)}\right]
:=-(n-1)\tilde{K}.
\]
Using the gradient estimate trick developed in \cite{[LD]}, from
Theorem 2.1 in \cite{[Wu]}, we know that there exists a
constant $C$ depending on $m_0$ and $n$ such that
\[
\frac{|\nabla f(x)|^2}{f^2(x)}
\leq\frac{[2(m_0-1)(n-1)+\epsilon]\tilde{K}}{2-\epsilon}
+C\cdot\left(\frac{1+\epsilon^{-1}}{R^2}+\lambda\right)
\]
for all $x\in B_p(R)$ and for any $0<\epsilon<2$. If the
positive function $f$ is defined on $(M^n,g)$, then
\[
|\nabla \ln f|^2\leq\frac{(m-1)(n-1)\tilde{K}}{2}-\lambda
+\sqrt{\frac{(m_0-1)^2(n-1)^2\tilde{K}^2}{4}-(m_0-1)(n-1)\tilde{K}\lambda}
\]
and
\[
\lambda\leq\frac{(m_0-1)(n-1)\tilde{K}}{4}.
\]
Substituting $\tilde{K}=K+\frac{\theta^2}{(m_0-n)(n-1)}$ into the
above inequalities yields \eqref{lemm1a},
\eqref{lemm1b} and \eqref{eigeestup}.
\end{proof}

Below, Theorem \ref{mainthm} can be easily obtained by Theorem \ref{L301}.
\begin{proof}[Proof of Theorem \ref{mainthm}]
Let $(M^n,g)$ be an $n$-dimensional ($n\geq2$) complete Riemannian
manifold and $\varphi \in C^2(M^n)$ be a function satisfying
$|\nabla \varphi|\leq\theta$. Assume that the $\infty$-dimensional
Bakry-\'{E}mery Ricci curvature satisfies
\[
Ric(L)\geq-(n-1).
\]
Let $\lambda_1$ be the first nontrivial eigenvalue of the diffusion
operator $L=\Delta-\nabla\varphi\cdot\nabla$. Hence $\lambda_1$
satisfies the equation
\[
L f=-\lambda_1 f,
\]
where $f$ is the eigenfunction. Now let $\lambda=\lambda_1$
and $K=1$ in Theorem \ref{L301}. Then
estimate \eqref{eigeestup} and Remark \ref{re22}
give the complete proof of Theorem \ref{mainthm}.
\end{proof}
% ------------------------------------------------------------------------

\section{Applications of Theorem \ref{L301} to Ricci solitons}\label{sec5}
In this section, we discuss the first eigenvalues of the diffusion operator
$L$ on gradient Ricci solitons by Theorem \ref{L301}. Recall that a
complete smooth Riemannian manifold $(M^n,g)$ is called a gradient Ricci
soliton if equation \eqref{soliton} holds. The gradient Ricci solitons are
called shrinking, steady and expanding accordingly when $\rho>0$,
$\rho=0$ and $\rho<0$.

On the gradient steady Ricci solitons $Ric+\nabla\nabla \varphi=0$,
there exists a positive constant $a>0$ (see \cite{[Ha]} or \cite{[Cao0]})
such that
\begin{equation}\label{norm}
R+|\nabla \varphi|^2=a^2.
\end{equation}
\begin{theorem}\label{401a}
On gradient steady Ricci solitons $Ric+\nabla\nabla \varphi=0$,
normalized as in \eqref{norm}, the first eigenvalue of diffusion operator $L$
satisfies $\lambda_1\leq {a^2}/{4}$.
\end{theorem}
\begin{remark}
A stronger result for gradient steady Ricci solitons has been proved by
Munteanu and Wang (Proposition 2.4 in \cite{[MuWa]})
using a different approach.
\end{remark}

\begin{proof}
Since $R\geq 0$ for any complete gradient steady Ricci soliton
\cite{[Chen]}, by \eqref{norm} we have $|\nabla \varphi|\leq a$.
Hence by Theorem \ref{L301}, in \eqref{eigeestup}, letting
$K=0$ and $m_0\to\infty$, we have the desired result.
\end{proof}

On the gradient shrinking Ricci solitons $Ric+\nabla\nabla \varphi=\rho g$, $\rho>0$,
there exists a constant $C$ (see \cite{[Ha]} or \cite{[Cao0]}) such that
\[
R+|\nabla \varphi|^2-2\rho \varphi=C.
\]
We normalize $\varphi$ (by adding a constant) so that $C=0$. Hence
\begin{equation}\label{norm2}
|\nabla\varphi|^2=-R+2\rho\varphi.
\end{equation}
If we let $\varphi\leq b$ for some positive constant $b$,
then we have
\begin{theorem}\label{401b}
On gradient shrinking Ricci solitons $Ric+\nabla\nabla \varphi=\rho g$, $\rho>0$,
normalized as in \eqref{norm2}, if $\varphi\leq b$ for some positive constant $b$, then
the first eigenvalue of diffusion operator $L$ satisfies $\lambda_1\leq \frac{\rho b}{2}$.
\end{theorem}

\begin{proof}
Since $R\geq 0$ for any gradient shrinking Ricci soliton \cite{[Chen]}, by
\eqref{norm2} we have $|\nabla\varphi|^2\leq 2\rho\varphi$. Since
we assume $\varphi\leq b$, $|\nabla \varphi|\leq \sqrt{2\rho b}$.
Hence by Theorem \ref{L301}, in \eqref{eigeestup}, letting
$K=0$ and $m_0\to\infty$, we have $\lambda_1\leq \frac{\rho b}{2}$.
\end{proof}
\begin{remark}
As we all know, for complete noncompact gradient shrinking Ricci solitons,
function $\varphi$ is quadratic growth of the distance function
unless it is trivial \cite{[Cao2]}. Hence our theorem is only used to
the compact case.
\end{remark}

\begin{remark}
Recently, A. Futaki, H.-Z. Li and X.-D. Li \cite{[FLL]} proved a lower 
bound estimate for the first non-zero eigenvalue of the diffusion 
operator on compact Riemannian manifolds. As an application, 
they derived a lower bound estimate for the diameter of compact 
gradient shrinking Ricci solitons.
\end{remark}

If we assume $\varphi>n/2$ on compact shrinking Ricci solitons,
we also have an upper bound of the first eigenvalue. Indeed,
tracing shrinking Ricci solitons yields
\[
R+\Delta \varphi=n\rho.
\]
Combining this with \eqref{norm2} we have
\[
\Delta_\varphi \varphi=n\rho-2\rho\varphi.
\]
Letting $\tilde{\varphi}=\varphi-n/2>0$, then
\[
\Delta_{\varphi} \tilde{\varphi}=-2\rho\tilde{\varphi}.
\]
Hence by the definition of $\lambda_1$, we conclude that $\lambda_1\leq2\rho$.

\section{Appendix}
In this section we shall give a detailed proof of Theorem \ref{L301}.
The proof nearly follows from that of Theorem 2.1 in \cite{[Wu]}.
We include it because we feel it might be useful in other applications.
For example, the course of proving gradient estimate is an important
step for proving splitting type theorems with $\infty$-dimensional
Bakry-\'{E}mery Ricci curvature (see \cite{[Wusp]}).
\begin{proof}[Detailed proof of Theorem \ref{L301}]
\textbf{Step 1.} The proof spirit is the same as the arguments
used in the proofs of Lemma 2.1, Theorem 2.2 and Theorem 1.3 in
\cite{[LD]}. Here our proof exactly follows from that
of Theorem 2.1 in \cite{[Wu]} (see also \cite{[Li2]})
with little modification. We firstly prove the following
inequality \eqref{lemmin1}. Define $h:=\ln f$. Then
$Lh=-|\nabla h|^2-\lambda$. Direct calculation shows that
\begin{equation}\label{jisuan2}
L|\nabla h|^2
=2h_{ij}^2+2(R_{ij}+\nabla^2\varphi)h_ih_j-2\langle\nabla h,\nabla
|\nabla h|^2\rangle.
\end{equation}
Choose a local orthonormal frame $\{e_1, e_2,...,e_n\}$ near any
such given point so that at the given point $\nabla h=|\nabla
h|e_1$. Then we can write
\begin{equation}
\begin{aligned}\label{jisuan3}
\left|\nabla\left|\nabla h\right|^2\right|^2=4\sum^n_{j=1}
\left(\sum^n_{i=1}h_ih_{ij}\right)^2
=4h^2_1\cdot\sum^n_{i=1}h^2_{1i}=4\left|\nabla
h\right|^2\cdot\sum^n_{i=1}h^2_{1i}.
\end{aligned}
\end{equation}

On the other hand, we carefully estimate the term $h_{ij}^2$ in
\eqref{jisuan2}. Notice
\begin{equation*}
\begin{aligned}
h_{ij}^2&\geq h^2_{11}+2\sum^n_{\alpha=2}h^2_{1\alpha}
+\sum^n_{\alpha=2}h^2_{\alpha\alpha}\\
&\geq h^2_{11}+2\sum^n_{\alpha=2}h^2_{1\alpha}
+\frac{1}{n-1}\left(\sum^n_{\alpha=2}h_{\alpha\alpha}\right)^2\\
&=h^2_{11}+2\sum^n_{\alpha=2}h^2_{1\alpha}
+\frac{1}{n-1}\left(\Delta h-h_{11}\right)^2\\
&=h^2_{11}+2\sum^n_{\alpha=2}h^2_{1\alpha}
+\frac{1}{n-1}\left(\left|\nabla h\right|^2+\lambda+h_{11}-\varphi_ih_i\right)^2\\
&\geq h^2_{11}+2\sum^n_{\alpha=2}h^2_{1\alpha}
+\frac{1}{n-1}\left[\frac{(|\nabla
h|^2+\lambda+h_{11})^2}{1+\frac{m_0-n}{n-1}}
-\frac{(\varphi_ih_i)^2}{\frac{m_0-n}{n-1}}\right]
\end{aligned}
\end{equation*}
for any constant $m_0(\geq n)$. Since $|\nabla \varphi|\leq \theta$, we have
\begin{equation}\label{jisuan4}
h_{ij}^2\geq\frac{m_0}{m_0-1}\sum^n_{i=1}h^2_{1i}+\frac{(|\nabla
h|^2+\lambda)^2}{m_0-1}+\frac{2h_{11}(|\nabla h|^2+\lambda)}{m_0-1}
-\frac{\theta^2 |\nabla h|^2}{m_0-n}.
\end{equation}
Note that $2h_{11}=\langle\nabla|\nabla h|^2,\nabla h\rangle\cdot|\nabla
h|^{-2}$. Substituting this into \eqref{jisuan4} yields
\begin{equation}\label{jisuan6}
h_{ij}^2\geq\frac{m_0}{m_0-1}\sum^n_{i=1}h^2_{1i}+\frac{(|\nabla
h|^2+\lambda)^2}{m_0-1}+\frac{(|\nabla
h|^2+\lambda)}{m_0-1}\cdot\frac{\langle\nabla|\nabla h|^2,\nabla
h\rangle}{|\nabla h|^2}-\frac{\theta^2 |\nabla h|^2}{m_0-n}.
\end{equation}
Putting \eqref{jisuan6}, \eqref{jisuan3} and \eqref{jisuan2}
together, we deduce
\begin{equation*}
\begin{aligned}
L\left|\nabla h\right|^2&\geq\frac{2m_0}{m_0-1}\sum^n_{i=1}h^2_{1i}
+\frac{2(|\nabla h|^2+\lambda)^2}{m_0-1}+\frac{2(|\nabla
h|^2+\lambda)}{m_0-1}\cdot\frac{\langle\nabla|\nabla h|^2,\nabla
h\rangle}{|\nabla h|^2}\\
&\,\,\,\,\,\,-\frac{2\theta^2 |\nabla
h|^2}{m_0-n}+2\left(R_{ij}+\nabla_i\nabla_j\varphi\right)h_ih_j
-2\left\langle\nabla h,\nabla|\nabla h|^2\right\rangle\\
&=\frac{m_0}{2(m_0-1)}\frac{|\nabla|\nabla h|^2|^2}{|\nabla
h|^2}+\frac{2(|\nabla
h|^2+\lambda)^2}{m_0-1}+2Ric(L)(\nabla h,\nabla h)\\
&\,\,\,\,\,\,-\frac{2\theta^2 |\nabla h|^2}{m_0-n}+\left[\frac{2
\lambda}{(m_0-1)|\nabla
h|^2}-\frac{2m_0-4}{m_0-1}\right]\cdot\left\langle\nabla|\nabla
h|^2,\nabla h\right\rangle.
\end{aligned}
\end{equation*}
Using $Ric(L)\geq-(n-1)K$, the function $h:=\ln f$ satisfies
\begin{equation}
\begin{aligned}\label{lemmin1}
L|\nabla h|^2&\geq\frac{m_0}{2(m_0-1)}\frac{|\nabla|\nabla
h|^2|^2}{|\nabla h|^2}+\frac{2\left(|\nabla h|^2
+\lambda\right)^2}{m_0-1}-2\left[(n-1)K+\frac{\theta^2}{m_0-n}\right]|\nabla h|^2\\
&\,\,\,\,\,\,+\left[\frac{2\lambda}{(m_0-1)|\nabla
h|^2}-\frac{2m_0-4}{m_0-1}\right]\cdot\left\langle\nabla|\nabla
h|^2,\nabla h\right\rangle
\end{aligned}
\end{equation}
for all $x\in B_p(R)$.

\textbf{Step 2.} To obtain estimates \eqref{lemm1a} and
\eqref{lemm1b}, we apply the diffusion operator $L$ to a
suitable function, and then use the maximum principle argument.

(i) Now we introduce a cut-off function. Let
\[
\phi(x):=\eta\left(\frac{\rho(x)}{R}\right),
\]
where $\eta(t)$ is a non-negative cut-off function such that
$\eta(t)\equiv 1$ for $0\leq t\leq1$, $\eta(t)\equiv 0$ for $t\geq
2$ and $0\leq\eta(t)\leq 1$ for $1<t<2$. Furthermore, take the
derivatives of $\eta$ to satisfy
\[
-C\eta^{1/2}(r)\leq\eta'\leq0 \quad\mathrm{and}\quad
\eta''(r)\geq-C,
\]
where $0<C<\infty$ is a universal constant.
Here $\rho(x)$ denotes the distance from some fixed $p\in M^n$.
Using an argument of Calabi \cite{[Calabi]} (see also
\cite{[Cheng-Yau]} or \cite{[Li-Yau]}), we can assume without loss
of generality that $\phi(x) \in C^2(M^n)$ with support in $B_p(2R)$.

Since $Ric(L)\geq-(n-1)K$ and $\nabla\varphi\geq-\theta$, by the
weighted Laplacian comparison theorem (Theorem 1.1 (a) in
\cite{[WW]}),
\[
L\rho\leq(n-1)\sqrt{K}\coth(\sqrt{K}\rho)+\theta.
\]
Note that
\[
L\phi=\frac{\eta'L\rho}{R}+\frac{\eta''\left|\nabla
\rho\right|^2}{R^2}.
\]
According to the definition of $\eta$, the function $\phi$ satisfies
\begin{equation}\label{tiaojian2}
L\phi\geq-C_1\left[(\sqrt{K}+\theta)R^{-1}+R^{-2}\right],
\end{equation}
where $C_1$ is a constant, depending only on $n$ and $C$, and
\begin{equation}\label{tiaojian3}
\frac{|\nabla\phi|^2}{\phi}\leq \frac{C_2}{R^2},
\end{equation}
where $C_2$ is also a constant, depending only on $C$.

\vspace{0.5em}

(ii) Let $G:=\phi\cdot|\nabla h|^2$. Using inequality
\eqref{lemmin1}, we obtain
\begin{equation}
\begin{aligned}\label{jisuan8}
LG&=(L\phi)\cdot|\nabla h|^2+2\left\langle\nabla\phi,\nabla|\nabla
h|^2\right\rangle+\phi\cdot L|\nabla h|^2\\
&\geq\frac{L\phi}{\phi}\cdot G+2\frac{\langle\nabla \phi,\nabla
G\rangle}{\phi}-2\frac{|\nabla \phi|^2}{\phi^2}
G+\frac{m_0}{2(m_0-1)}\cdot\phi\cdot\frac{|\nabla|\nabla
h|^2|^2}{|\nabla h|^2}\\
&\,\,\,\,\,\,-2\left[(n-1)K+\frac{\theta^2}{m_0-n}\right]G
-\frac{2m_0-4}{m_0-1}\cdot\left\langle\nabla h,\nabla
G\right\rangle+\frac{2m_0-4}{m_0-1}\cdot\frac{\langle\nabla
h,\nabla \phi\rangle}{\phi} G\\
&\,\,\,\,\,\,+\frac{2\lambda\cdot\left\langle\nabla h,\nabla
G\right\rangle}{(m_0-1)|\nabla h|^2}-\frac{2\lambda
\cdot\left\langle\nabla h,\nabla \phi\right\rangle}{m_0-1}
+\frac{2}{m_0-1}\left(\phi^{-1}G^2+2\lambda G+\phi\lambda^2\right).
\end{aligned}
\end{equation}
In the following, we will estimate `bad' terms on the right hand
side (or RHS for short) of \eqref{jisuan8}. On one hand,
\begin{equation*}
\begin{aligned}
|\nabla G|^2&=\left|\nabla\left(\phi\cdot|\nabla h|^2\right)\right|^2\\
&=|\nabla\phi|^2\cdot|\nabla h|^4+2\phi|\nabla
h|^2\cdot\left\langle\nabla\phi,\nabla|\nabla
h|^2\right\rangle+\phi^2\left|\nabla|\nabla h|^2\right|^2\\
&=-\frac{|\nabla\phi|^2}{\phi^2}\cdot
G^2+2\frac{\langle\nabla\phi,\nabla G\rangle}{\phi}\cdot
G+\phi^2\left|\nabla|\nabla h|^2\right|^2.
\end{aligned}
\end{equation*}
This implies
\begin{equation}\label{huajian1}
\phi\cdot\frac{|\nabla|\nabla h|^2|^2}{|\nabla h|^2}=\frac{|\nabla
G|^2}{G}+\frac{|\nabla \phi|^2}{\phi^2}\cdot
G-2\frac{\langle\nabla\phi,\nabla G\rangle}{\phi}.
\end{equation}
On the other hand, we can easily have
\begin{equation}\label{huajian2}
2\frac{\langle\nabla h,\nabla \phi\rangle}{\phi} G\geq-2|\nabla
\phi|\phi^{-\frac32}G^{\frac32}
\quad \mathrm{and} \quad
-2\langle\nabla h,\nabla \phi\rangle\geq-2|\nabla
\phi|\phi^{-\frac12}G^{\frac12}.
\end{equation}
Substituting \eqref{huajian1} and \eqref{huajian2} into the
RHS of \eqref{jisuan8} gives
\begin{equation}
\begin{aligned}\label{jisuan9}
LG&\geq\frac{L\phi}{\phi}\cdot G+\frac{m_0-2}{m_0-1}\cdot
\frac{\langle\nabla\phi,\nabla G\rangle}{\phi}
-\frac{3m_0-4}{2(m_0-1)}\cdot\frac{|\nabla\phi|^2}{\phi^2}G\\
&\,\,\,\,\,\,+\frac{m_0}{2(m_0-1)}\cdot\frac{|\nabla G|^2}{G}
+\left[\frac{4\lambda}{m_0-1}-2(n-1)K-\frac{2\theta^2}{m_0-n}\right]G\\
&\,\,\,\,\,\,-\frac{2m_0-4}{m_0-1}\cdot\left\langle\nabla h,\nabla
G\right\rangle -\frac{2m_0-4}{m_0-1}\cdot|\nabla\phi|\phi^{-\frac32}
G^{\frac32}+\frac{2\lambda\cdot\left\langle\nabla h,\nabla
G\right\rangle}{(m_0-1)|\nabla h|^2}\\
&\,\,\,\,\,\,
-\frac{2\lambda}{m_0-1}\cdot|\nabla\phi|\phi^{-\frac12}G^{\frac12}
+\frac{2\phi^{-1}G^2}{m_0-1}+\frac{2\lambda^2\phi}{m_0-1}.
\end{aligned}
\end{equation}
Let $x_0\in B_p(2R)\subset M^n$ be a point where $G$ achieves a
maximum. By the maximum principle, we have
\begin{equation}\label{tiaojian1}
LG(x_0)\leq0,\quad\nabla G(x_0)=0.
\end{equation}

All further calculations in this proof will be at $x_0$. Multiplying
both sides of \eqref{jisuan9} by $(m_0-1)\phi$ and using
\eqref{tiaojian1}, then \eqref{jisuan9} reduces to
\begin{equation*}
\begin{aligned}
0\geq LG&\geq(m_0-1)L\phi\cdot G
-\frac{3m_0-4}{2}\cdot\frac{|\nabla\phi|^2}{\phi}G\\
&\,\,\,\,\,\,+\left[4\lambda-2(m_0-1)(n-1)K-\frac{2(m_0-1)\theta^2}{m_0-n}\right]\phi
G-(2m_0-4)\cdot|\nabla\phi|\phi^{-\frac12} G^{\frac32}\\
&\,\,\,\,\,\,-2\lambda\cdot|\nabla\phi|\phi^{\frac12}
G^{\frac12}+2G^2+2\lambda^2\phi^2.
\end{aligned}
\end{equation*}
Combining this with the above estimates of $\phi$ (
\eqref{tiaojian2} and \eqref{tiaojian3}), we get
\begin{equation}
\begin{aligned}\label{jisuan10}
0&\geq-\left(C_3R^{-1}(\sqrt{K}+\theta)+C_4R^{-2}\right)G
+\left[4\lambda\phi-2(m_0-1)(n-1)\phi K-\frac{2(m_0-1)\phi\theta^2}{m_0-n}\right]G\\
&\,\,\,\,\,\,-C_5R^{-1}G^{\frac32}-C_6\lambda
R^{-1}G^{\frac12}+2G^2+2\lambda^2\phi^2,
\end{aligned}
\end{equation}
where $C_3$ is some constant depending on $m_0$, $n$ and
$C$; $C_4$ and $C_5$ are fixed constants depending on
$m_0$ and $C$; and $C_6$ is also a constant, depending only on $C$.

\vspace{0.5em}

Since $x_0$ is the maximum point of the function $G$ and $\phi=1$ on
$B_p(R)$, hence
\[
\phi(x_0)|\nabla h|^2(x_0)\geq\sup_{B_p(R)}|\nabla h|^2(x).
\]
On the other hand, using the fact that
\[
\phi(x_0)|\nabla h|^2(x_0)\leq\phi(x_0)\sup_{B_p(2R)}|\nabla
h|^2(x),
\]
we conclude that $\sigma(R)\leq\phi(x_0)\leq1$,
where $\sigma(R)$ is defined by
\[
\sigma(R):=\frac{\sup_{B_p(R)}|\nabla h|^2(x)}{\sup_{B_p(2R)}|\nabla
h|^2(x)}.
\]
Applying this to \eqref{jisuan10} yields
\begin{equation}
\begin{aligned}\label{jisuan11}
0&\geq-\left[C_3R^{-1}(\sqrt{K}+\theta)+C_4R^{-2}
-4\lambda\sigma(R)+2(m_0-1)(n-1)K+\frac{2(m_0-1)\theta^2}{m_0-n}\right]G\\
&\,\,\,\,\,\,-C_5R^{-1}G^{\frac32}-C_6\lambda
R^{-1}G^{\frac12}+2G^2+2\lambda^2\sigma^2(R).
\end{aligned}
\end{equation}
Below we want to estimate some `bad' terms of the RHS of
\eqref{jisuan11}. Using the Schwarz inequality, we have the
following three inequalities:
\[
-C_5R^{-1}G^{\frac32}\geq-\epsilon
G^2-\frac{C^2_5}{4}\epsilon^{-1}R^{-2}G,
\]
\[
-C_6\lambda R^{-1}G^{\frac12}\geq-\epsilon
\lambda^2-\frac{C^2_6}{4}\epsilon^{-1}R^{-2}G
\]
and
\[
-C_3R^{-1}\sqrt{K}\geq-\epsilon K
-\frac{C^2_3}{4}\epsilon^{-1}R^{-2}
\]
for all $\epsilon>0$. Hence at the point $x_0$,
\eqref{jisuan11} can be rewritten as
\begin{equation}
\begin{aligned}\label{jisuan12}
0&\geq-\left[C_7(1+\epsilon^{-1})R^{-2}
-4\lambda\sigma(R)+\left(2(m_0-1)(n-1)+\epsilon\right)K
+\frac{2(m_0-1)\theta^2}{m_0-n}\right]G\\
&\,\,\,\,\,\,+(2-\epsilon)G^2+2\lambda^2\sigma^2(R)
-\epsilon\lambda^2,
\end{aligned}
\end{equation}
where $C_7$ is some constant, depending only on $m_0$,
$n$ and $C$. Now we have a quadratic inequality in $G$.
If $\epsilon<2$, then by \eqref{jisuan12} we get
\begin{equation}\label{diaojiana}
A^2-4(2-\epsilon)\lambda^2 (2\sigma^2(R)-\epsilon)\geq0,
\end{equation}
and an upper bound
\begin{equation}\label{estimate1}
G(x_0)\leq\frac{A+\sqrt{A^2-4(2-\epsilon)\lambda^2
\big(2\sigma^2(R)-\epsilon\big)}}{2(2-\epsilon)},
\end{equation}
where $A:=C_7(1+\epsilon^{-1})R^{-2}-4\lambda\sigma(R)
+\left[2(m_0-1)(n-1)+\epsilon\right]K+\frac{2(m_0-1)\theta^2}{m_0-n}$.

We will see that \eqref{diaojiana} and \eqref{estimate1} imply
estimates \eqref{lemm1a}, \eqref{lemm1b} and \eqref{eigeestup} in
Theorem \ref{L301}. In fact for any $x\in B_p(R)$,
\[
|\nabla h|^2(x)=\phi(x)|\nabla h|^2(x)\leq G(x_0).
\]
Combining this with \eqref{estimate1} and noticing that
$0\leq\sigma(R)\leq1$, we have that
\begin{equation}
\begin{aligned}\label{estimate2}
\frac{|\nabla f|^2}{f^2}(x)
&\leq\frac{2A+\sqrt{4(2-\epsilon)\lambda^2
\big(2\sigma^2(R)-\epsilon\big)}}{2(2-\epsilon)}\\
&\leq\frac{[2(m_0-1)(n-1)+\epsilon]K}{(2-\epsilon)}
+\frac{2(m_0-1)\theta^2}{(m_0-n)(2-\epsilon)}
+\tilde{C}\left[(1+\epsilon^{-1})R^{-2}+\lambda\right]
\end{aligned}
\end{equation}
for all $x\in B_p(R)$ and for any $0<\epsilon\leq2\sigma^2(R)<2$,
where $\tilde{C}$ is a constant, depending only on $m_0$, $n$ and
$C$. Hence the proof of the estimate \eqref{lemm1a} is finished.

\vspace{0.5em}

In the following we assume that $f$ is defined on $M^n$. For any
$0<\epsilon<2$, if we take $R\to\infty$ in \eqref{estimate1}, then
$\sigma(R)\to1$ and \eqref{estimate1} becomes
\begin{equation*}
\begin{aligned}
|\nabla h|^2(x)&\leq\frac{\left[-4\lambda+\Big(2(m_0-1)(n-1)
+\epsilon\Big)K\right]}{{2(2-\epsilon)}}+\frac{(m_0-1)\theta^2}{(m_0-n)(2-\epsilon)}\\
&\,\,\,\,\,\,+\frac{\sqrt{\left[-4\lambda
+\Big(2(m_0-1)(n-1)+\epsilon\Big)K+\frac{2(m_0-1)\theta^2}{m_0-n}\right]^2
-4(2-\epsilon)^2\lambda^2}}{2(2-\epsilon)}.
\end{aligned}
\end{equation*}
Letting $\epsilon\rightarrow0+$, we obtain \eqref{lemm1b}.

At last, since $f$ is defined on $M^n$, $\sigma(R)\to1$ in
\eqref{diaojiana} as $R\to\infty$. Then taking
$\epsilon\rightarrow0$, the inequality \eqref{diaojiana} becomes
\[
\frac{\left[(m_0-1)(n-1)K+\frac{m_0-1}{m_0-n}\theta^2\right]^2}{4}
-\left[(m_0-1)(n-1)K+\frac{m_0-1}{m_0-n}\theta^2\right]\lambda\geq0.
\]
Thus
\[
\lambda\leq\frac{(m_0-1)(n-1)K+\frac{m_0-1}{m_0-n}\theta^2}{4}.
\]
This finishes the proof of Theorem \ref{L301}.
\end{proof}

\section*{Acknowledgment}
The author would like to express his gratitude to the referee 
for pointing out a simple proof of Theorem \ref{L301}, as 
well as the optimization of eigenvalue estimate of Theorem
\ref{mainthm} (see Remark \ref{rem11b}).

% ------------------------------------------------------------------------
\bibliographystyle{amsplain}

\end{document}